\documentclass[12pt]{article}
\usepackage{a4,latexsym,amsmath,amssymb,amsthm,enumerate,xcolor,eucal,tikz,
tabularx}

\usetikzlibrary{calc}

\newtheorem{theorem}{Theorem}
\newtheorem{lemma}[theorem]{Lemma}
\newtheorem{proposition}[theorem]{Proposition}

\newtheorem{observation}[theorem]{Observation}
\newtheorem{conjecture}[theorem]{Conjecture}

\newtheorem{phase}{Phase}

\newcommand\Setx[1] {\left\{{#1}\right\}}
\newcommand\Set[2] {\left\{{#1}:\,{#2}\right\}}
\newcommand\map[3] {{#1}:\,{#2}\to{#3}}
\newcommand\size[1] {\left|{#1}\right|}
\newcommand{\BB}{\mathcal B}
\newcommand{\PP}{\mathcal P}

\newcommand{\WW}{\mathcal W}
\newcommand{\cprob}[2]{\mathbf{P}\left(#1\ |\ #2\right)}
\newcommand{\prob}[1]{\mathbf{P}\left(#1\right)}

\newenvironment{tikzgraphv}[1]
  {\begin{tikzpicture}
      [vertex/.style={circle, black, fill, inner sep=0mm, minimum
        size=#1},
      whitev/.style={circle, draw=black, fill=white, thick, inner sep=0mm, minimum
        size=#1},       
      every edge/.style={thick,draw,>=stealth},
      lbl/.style={text height=10pt,below}
      ]\begin{scope}}
  {\end{scope}\end{tikzpicture}}
\newenvironment{tikzgraph}{\begin{tikzgraphv}{4pt}}{\end{tikzgraphv}}
\newenvironment{conflict}{\begin{tikzgraphv}{6pt}[scale=0.8]}{\end{tikzgraphv}}

\usepackage{subfigure}
\newbox\subfigbox


\title{\textbf{Fractional total colourings\\of graphs of high
    girth\footnote{Supported by project GA\v{C}R 201/09/0197 of the
      Czech Science Foundation.}}} %
\author{Tom\'{a}\v{s} Kaiser\footnote{Department of Mathematics and
    Institute for Theoretical Computer Science, University of West
    Bohemia, Univerzitn\'{\i}~8, 306~14~Plze\v{n}, Czech
    Republic. E-mail: \texttt{kaisert@kma.zcu.cz}. Supported by
    Research Plan MSM 4977751301 of the Czech
    Ministry of Education.}\\
  Andrew King\footnote{Institute for Theoretical Computer Science,
    Faculty of Mathematics and Physics, Charles University,
    Malostransk\'{e} n\'{a}m\v{e}st\'{\i} 25, 118 00 Prague, Czech
    Republic. E-mail: \texttt{aking6@cs.mcgill.ca}.}\\
  Daniel Kr\'{a}l'\footnote{Institute for Theoretical Computer
    Science, Faculty of Mathematics and Physics, Charles University,
    Malostransk\'{e} n\'{a}m\v{e}st\'{\i} 25, 118 00 Prague, Czech
    Republic. E-mail: \texttt{kral@kam.mff.cuni.cz}. The Institute for
    Theoretical Computer Science is supported by the Ministry of
    Education of the Czech Republic as project 1M0545.}}

\date{}

\begin{document}
\maketitle

\begin{abstract}
  Reed conjectured that for every $\epsilon>0$ and $\Delta$ there
  exists $g$ such that the fractional total chromatic number of a
  graph with maximum degree $\Delta$ and girth at least $g$ is at most
  $\Delta+1+\epsilon$. We prove the conjecture for $\Delta=3$ and for
  even $\Delta\geq 4$ in the following stronger form: For each of
  these values of $\Delta$, there exists $g$ such that the fractional
  total chromatic number of any graph with maximum degree $\Delta$ and
  girth at least $g$ is equal to $\Delta+1$.
\end{abstract}

\section{Introduction}

Total colouring and edge colouring share many common features. For
instance, Vizing's theorem asserts that the chromatic index of any
graph with maximum degree $\Delta$ is at most $\Delta+1$. The total
chromatic number of such a graph is known to be at most $\Delta+C$,
where $C$ is a constant, and is conjectured to be at most
$\Delta+2$. Asymptotically, these bounds are far from the trivial
upper bounds of $2\Delta-1$ and $2\Delta$, respectively.

In other ways, however, the two notions behave differently. Consider
their fractional versions (see below for the necessary
definitions). It is known that the fractional chromatic index of a
cubic bridgeless graph is equal to 3, the obvious lower bound. The
analogous assertion for fractional total colouring is false, as shown
by the graph $K_4$, whose fractional total chromatic number is 5. One
might ask whether high girth makes the fractional total chromatic
number arbitrarily close to $\Delta+1$. Indeed,
Reed~\cite{bib:Ree-talk} conjectured that this is exactly the case
(see Conjecture~\ref{conj:reed} below). In this paper, we confirm the
conjecture for $\Delta = 3$ and for even $\Delta$, in a stronger form.

Before stating the result in detail, we introduce the relevant
terminology. Let $G$ be a graph. The vertex and edge sets of $G$ will
be denoted by $V(G)$ and $E(G)$. Let $w$ be a function assigning each
independent set $I$ of $G$ a real number $w(I) \in [0,1]$. The
\emph{weight} $w[x]$ of $x\in V(G)$ with respect to $w$ is defined as
the sum of $w(I)$ over all independent sets $I$ in $G$ containing
$x$.

The function $w$ is a \emph{fractional colouring} of $G$ if for each
vertex $v$ of $G$,
\begin{equation*}
  w[v] \geq 1.
\end{equation*}
The \emph{size} $\size w$ of a fractional colouring $w$ is the sum of
$w(I)$ over all independent sets $I$. The \emph{fractional chromatic
  number} $\chi_f(G)$ of $G$ is the infimum of $\size w$ as $w$ ranges
over fractional colourings of $G$. It is easy to see that $\chi_f(G)
\leq \chi(G)$. It is also known (see, e.g.,
\cite[p.~42]{bib:SU-fractional}) that $\chi_f(G)$ is rational and,
although it is defined as an infimum, there exists a fractional
colouring of size $\chi_f(G)$. Moreover, among the optimal fractional
colourings there exists a rational-valued one.

Fractional colourings may be viewed in several ways, each of which can
be useful in a different context. A basic observation concerning their
equivalence is given by the following lemma.
\begin{lemma}\label{l:basic}
  Let $G$ be a graph. The following are equivalent:
  \begin{enumerate}[\quad(i)]
  \item $\chi_f(G) \leq k$,
  \item there exists an integer $N$ and a multiset $\WW$ of $k\cdot N$
    independent sets in $G$, such that each vertex is contained in
    exactly $N$ sets of $\WW$,
  \item there exists a probability distribution $\pi$ on independent
    sets of $G$ such that for each vertex $v$, the probability that
    $v$ is contained in a random independent set (with respect to
    $\pi$) is at least $1/k$.\qed
  \end{enumerate}
\end{lemma}
For more details on fractional colouring, we refer the reader
to~\cite{bib:SU-fractional}.

The \emph{fractional chromatic index} $\chi'_f(G)$ of $G$ is defined
as the fractional chromatic number of the line graph $L(G)$. An
important result concerning this parameter follows from the work of
Edmonds~\cite{bib:Edm-maximum} (also see
Seymour~\cite{bib:Sey-multicolourings}):
\begin{theorem}\label{t:edge}
  The fractional chromatic index of a bridgeless cubic graph $G$
  equals 3. Equivalently, there is a multiset of $3N$ perfect
  matchings in $G$ such that each edge is contained in exactly $N$ of
  them.
\end{theorem}

The \emph{total graph} $T(G)$ of $G$ has vertex set $V(G)\cup E(G)$; a
pair $xy$ is an edge of $T(G)$ if one of the following holds:
\begin{itemize}
\item $x$ and $y$ are adjacent vertices of $G$,
\item $x$ is an edge of $G$ and $y$ is one of its endvertices,
\item $x$ and $y$ are incident edges of $G$.
\end{itemize}
 
Independent sets in $T(G)$ are called \emph{total independent sets} of
$G$. The \emph{total chromatic number} $\chi''(G)$ of $G$ is defined
as $\chi(T(G))$. Similarly, a \emph{fractional total colouring} of $G$
is simply a fractional colouring of $T(G)$, and we define the
\emph{fractional total chromatic number} $\chi''_f(G)$ of $G$ as
$\chi_f(T(G))$. 

Let us stress that when applying Lemma~\ref{l:basic} to total
fractional colourings, one has to work with total independent
sets. Thus, for instance, $\chi''_f(G) \leq k$ is equivalent to the
existence of $kN$ total independent sets in $G$ such that each vertex
and each edge are contained in $N$ of the sets.

Behzad~\cite{bib:Beh-graphs} and Vizing~\cite{bib:Viz-some}
independently conjectured the following upper bound on $\chi''(G)$:
\begin{conjecture}[Total colouring conjecture]
  \label{conj:total}
  For any graph with maximum degree $\Delta$, 
  \begin{equation*}
    \chi''(G)\leq\Delta+2.
  \end{equation*}
\end{conjecture}

Currently the best upper bound on $\chi''(G)$ in terms of the maximum
degree $\Delta$ of $G$ is due to Molloy and Reed~\cite{bib:MR-bound}
who proved that $\chi''(G)$ is bounded by $\Delta + C$ for a suitable
constant $C$.

Kilakos and Reed~\cite{bib:KR-fractionally} proved the analogue of
Conjecture~\ref{conj:total} for the fractional version of total
colouring:
\begin{theorem}
  \label{t:kilakos-reed}
  For any graph $G$ with maximum degree $\Delta$,
  \begin{equation*}
    \chi''_f(G) \leq \Delta+2.
  \end{equation*}
\end{theorem}
Recently, Ito et al.~\cite{bib:IKR-characterization} showed that the
only graphs $G$ with $\chi''_f(G) = \Delta+2$ are $K_{2n}$ and
$K_{n,n}$ ($n\geq 1$).

As mentioned above, Reed~\cite{bib:Ree-talk} conjectured that high
girth makes the fractional total chromatic number close to $\Delta+1$:
\begin{conjecture}\label{conj:reed}
  For every $\varepsilon > 0$ and every integer $\Delta$, there exists $g$
  such that the fractional chromatic number of any graph with maximum
  degree $\Delta$ and girth at least $g$ is at most $\Delta + 1 +
  \varepsilon$. 
\end{conjecture}

In the present paper, we prove a stronger form of the conjecture for
$\Delta = 3$ (we call graphs $G$ with maximum degree 3
\emph{subcubic}). The argument also applies for even $\Delta \geq 4$.

Our first main result is the following theorem:
\begin{theorem}
  \label{t:main}
  If $G$ is a subcubic graph of girth at least $15\,840$, then
  \begin{equation*}
    \chi''_f(G) = 4.
  \end{equation*}
\end{theorem}
As noted above, this confirms a particular case of
Conjecture~\ref{conj:reed}. In
Sections~\ref{sec:prob}--\ref{sec:cubic}, we first prove
Theorem~\ref{t:main} for graphs $G$ which are cubic and bridgeless. In
Section~\ref{sec:subcubic}, the result is extended to subcubic
graphs. Finally, in Section~\ref{sec:even}, we prove our second main
result:
\begin{theorem}\label{t:even}
  For any \emph{even} integer $\Delta$, there exists a constant
  $g(\Delta)$ such that if $G$ is a graph with maximum degree $\Delta$
  and girth at least $g(\Delta)$, then $\chi''_f(G) = \Delta+1$.
\end{theorem}

Following the initial submission of this paper, Kardo\v{s}, Kr\'{a}l'
and Sereni~\cite{bib:KKS-last} completed the proof of
Conjecture~\ref{conj:reed}, building on techniques developed here.


\section{Overview of the method}
\label{sec:overview}

We now present an overview of our method, restricting our attention to
the cubic bridgeless case. In the proof of Theorem~\ref{t:main}, the
required fractional total colouring is obtained indirectly, by
constructing a suitable probability distribution and using
Lemma~\ref{l:basic}.

To show that a cubic graph $G$ has $\chi_f''(G) = 4$, it suffices to
construct a probability distribution $\pi$ on total independent sets
such that each vertex and edge is included in a random total
independent set with probability at least $1/4$. Consider the set $Y$
consisting of a vertex of $G$ and the three edges incident with
it. Since any total independent set contains at most one object from
$Y$, we must ensure that every total independent set $T$ with $\pi(T)
> 0$ contains exactly one element of $Y$. We arrive at the following
definitions.

We will say that a set $X\subseteq V(G)\cup E(G)$ \emph{covers} a vertex
$v\in V(G)$ if $v\in X$ or $v$ is incident with an edge in $X$. A set
covering every vertex is \emph{full}. The set of full total independent
sets of $G$ will be denoted by $\Phi(G)$. 

For the reason outlined above, $\pi$ will assign nonzero probability
to full sets only. Under this provision, it is clear that if each
$x\in V(G)\cup E(G)$ has the same probability of being included in a
$\pi$-random total independent set, then $\chi''_f(G) = 4$.

The distribution is constructed by means of a probabilistic algorithm
described in Section~\ref{sec:algorithm}. The algorithm produces a
full total independent set $\tilde T$ for any given choice of an
(oriented) 2-factor $F$ in $G$. It will be observed that for a fixed
choice of $F$, all the edges of $F$ have the same chance of being
included in $\tilde T$. The probability of inclusion in $\tilde T$ is
also constant on the edges not in $F$, as well as on the vertices of
$G$. To ensure that the edges in $F$ get the same probability as those
not in $F$, we `average' using Theorem~\ref{t:edge} which guarantees
the existence of a multiset $\WW$ of perfect matchings such that every
edge is contained in one third of the members of $\WW$. By running the
algorithm with $F$ ranging over complements of all the perfect
matchings from $\WW$ and taking the average of the distributions thus
produced, we indeed make the probability constant on all of $E(G)$. It
will also be constant on $V(G)$, but the two constant values will not
be the same. Luckily, we will observe (using the results of
Section~\ref{sec:prob}) that the probability of inclusion for a vertex
is higher than for an edge. This will enable us to augment the
distribution to the desired one, essentially by taking a weighted
average with a distribution on perfect matchings obtained from
Theorem~\ref{t:edge}.

In the remainder of this section, we introduce some more notation and
terminology for later use.

An \emph{oriented 2-factor} in a graph $G$ is a 2-factor with a
specified orientation of each of its cycles. Assume an oriented
2-factor is chosen. For $v\in V(G)$, $v^-$ and $v^+$ denote the
precedessor and successor of $v$ on $F$ with respect to the given
orientation of $F$. Similarly, if $e\in E(F)$, then $e^-$ is the edge
that precedes $e$ on $F$ and $e^+$ is the edge that follows it. The
\emph{left} (\emph{right}) end of an edge or a subpath of $F$ is its
first (last) vertex with respect to the given orientation. 

A path with endvertices $u$ and $v$ will also be referred to as a
\emph{$uv$-path}.

We will occasionally need to speak about the distance between two
edges $e$ and $f$ of $G$. This is defined as the distance between $e$
and $f$ in the total graph $T(G)$. The distance of a vertex from an
edge is defined similarly. In particular, note that the distance
between an edge and its endvertex is 1.

For an integer $i$, we define the \emph{$i$-neighbourhood} $N_i(e)$ of
an edge $e\in E(G)$ as the set of all the vertices of $G$ whose
distance from $e$ is at most $i$, and all the edges with both
endvertices in $N_i(e)$. If $B$ is a set of edges, then $N_i(B)$ is
the union of all $N_i(e)$ as $e$ ranges over $B$.


\section{A recurrence}
\label{sec:prob}

The purpose of this section is to analyse two sequences of real
numbers, $p_k(i)$ and $q_k(i)$, needed later in
Section~\ref{sec:algorithm}. In that section, we will present an
algorithm that constructs a random total independent set $T$ in a
graph $G$ whose vertices and edges are divided into $k$ `levels'. It
will eventually turn out that the probability of the inclusion of a
vertex (edge, respectively) $x$ in the resulting total independent set
is $q_k(i)$ ($p_k(i)$, respectively), conditioned on $x$ being at
level $i$. We postpone the details to the next section.

Let $k$ be a positive integer. For $i = 1,\dots,k$, we define the
values $p_k(i)$ and $q_k(i)$ by the recurrence
\begin{align}
  2p_k(i) + q_k(i) &= 1,\notag\\
  q_k(i) &= p_k(i) \left( 1 - \frac1k - \frac1k\sum_{j=1}^{i-1}
    p_k(j)\right).\label{eq:q}
\end{align}
Observe that $p_k(1) = k/(3k-1)$ and $q_k(1) = (k-1)/(3k-1)$. We set
\begin{equation*}
  p_k^* = \sum_{i=1}^k \frac{p_k(i)}k \text{\quad and\quad}
  q_k^* = \sum_{i=1}^k \frac{q_k(i)}k\ .
\end{equation*}

We want to understand the values of $p_k(i)$ and $q_k(i)$ as $k$
becomes very large. In particular, we will need to know that $q_k^*
\geq 1/4$ for large enough $k$. It suffices to prove the following.

\begin{lemma}\label{l:sqrt}
  We have
  \begin{equation*}
    \lim_{k\to\infty} p_k^*= 3 - \sqrt{7}.
  \end{equation*}
\end{lemma}

\begin{proof}
  For each $k\geq 1$, consider the piecewise linear function $h_k(x)$
  on the interval $[0,1]$ satisfying
  \begin{equation*}
    h_k\left(\frac{i-1}{k-1}\right) = p_k(i)
  \end{equation*}
  for $i=1,\dots,k$, and linear on each interval
  $[\tfrac{i-1}{k-1},\tfrac{i}{k-1}]$. It can be shown that for fixed
  $x\in [0,1]$, the sequence $(h_k(x))_{k=1}^\infty$ converges; we
  define $f(x)$ to be its limit. By the Arzel\`{a}--Ascoli theorem
  (see, e.g.,~\cite[p.~169]{bib:Roy-real}), the resulting function $f$
  on $[0,1]$ is continuous and the convergence of $h_k$ to $f$ is
  uniform.

  The sum $p_k^*$ can be viewed as a Riemann sum which approaches
  $\int_0^1\! f(x)\,dx$ as $k$ tends to infinity. Combining the two
  equations in \eqref{eq:q}, we obtain
  \begin{equation*}
    p_k(i) = \frac{k}{3k-1-\sum_{j=1}^{i-1}p_k(j)},
  \end{equation*}
  which implies that consecutive values of $p_k$ are related by the
  equation
  \begin{equation*}
    \frac 1{p_k(i)}- \frac 1{p_k(i+1)} = \frac{p_k(i)}k.
  \end{equation*}
  From this, we compute
  \begin{equation*}
    p_k(i+1)-p_k(i) = \frac{p_k(i)^3}{k-p_k(i)^2}.
  \end{equation*}
  In the limit, as $k\to\infty$, $p_k(i+1)-p_k(i)$ approximates
  $f'(x)/k$.  Thus $f(x)$ satisfies the differential equation
  \begin{equation*}
    f'(x)=\lim_{k\to\infty}\frac{f(x)^3}{1-f(x)^2/k}= f(x)^3.
  \end{equation*}
  In view of the observation that $p_k(1) = k/(3k-1)$, which leads
  to the initial condition $f(0)=1/3$, the solution to this
  differential equation is $f(x)=(9-2x)^{-1/2}$.  The result follows
  immediately, since
  \begin{equation*}
    \lim_{k\to\infty} p^*_k = 
    \lim_{k\to\infty}\sum_{i=1}^{k}\frac{p_k(i)}k=\int_0^1\! f(x)\, dx = 
    \left[-\sqrt{9-2x}\right]_0^1 = 3-\sqrt7.
  \end{equation*}
\end{proof}


\section{An algorithm}
\label{sec:algorithm}

Let $G$ be a cubic graph. Throughout this and the following section,
we assume that $G$ has girth at least $15k\ell$, where $k$ and $\ell$
are sufficiently large integers which will be determined in the proof
of Lemma~\ref{l:weights2}. The notation $p_k(i)$ and $q_k(i)$ of
Section~\ref{sec:prob} will be abbreviated to $p(i)$ and $q(i)$ as $k$
is fixed throughout the exposition.

Fix an oriented 2-factor $F$. A set $B\subseteq E(G)$ will be said to
be \emph{$r$-distant} (where $r$ is an integer) if the distance
between any two of its edges in $G$ is at least $r$. Furthermore, $B$
is \emph{$(F,\ell)$-sparse} if it is 4-distant and $F-B$ consists of
paths whose length is at least $\ell$ and at most $7\ell$. Observe
that by the above assumptions, each cycle of $F$ contains at least two
edges from any $(F,\ell)$-sparse set.

Let $B\subseteq E(F)$ be an $(F,\ell)$-sparse set of edges. In this
section, we describe a probabilistic algorithm producing a full total
independent set $\tilde T = \tilde T(F,B)$.

The \emph{mate} $v^*$ of a vertex $v\in V(G)$ is the neighbour of $v$
in $G-E(F)$. The edges in $B$ will be referred to as \emph{boundary
  edges}.

\begin{phase}
  We construct an intermediate set $T = T(F,B) \subseteq V(T(G))$ with
  the property that for any component $P$ of $F-B$, the vertices and
  edges of $T$ contained in $P$ constitute a total independent set
  (although $T$ as a whole need not be total independent).
\end{phase}

Make a uniformly random choice of a function $\map \lambda B
{\Setx{1,\dots,k}}$, assigning a \emph{level} $\lambda(e)$ to each
edge $e\in B$. The notion of a level is extended to each vertex or
edge $x\in V(F)\cup E(F)$ by defining $\lambda(x)$ to be the level of
the closest boundary edge in the direction opposite to the prescribed
orientation of $F$. If $Q$ is a component of $F-B$, we define
$\lambda(Q)$ as the level of any vertex of $Q$.

Let $e^1,\dots,e^m$ be an ordering of the boundary edges such that
$\lambda(e^i) \leq \lambda(e^j)$ if $i < j$.

We construct the set $T$ in a sequence of steps, starting with
$T=\emptyset$. At step $i$ ($1\leq i \leq m$), we process the boundary
edge $e^i$ and the path $P^i$ of $F-B$ following $e^i$ (with respect
to the selected orientation of $F$). Enumerate the vertices and edges
of $P^i$ as $u^i_0,e^i_1,u^i_1,e^i_2,\dots,u^i_r$, where the order of
the vertices $u^i_j$ and the edges $e^i_j$ is again based on the
orientation of $F$. To make the notation more uniform, we may write
$e^i = e^i_0$. In the following discussion, we drop the superscript
$i$.

\begin{sloppypar}
  For the purpose of the description below, we consider the endvertex
  of $e_0$ different from $u_0$ to be a new \emph{virtual} vertex
  $u_{-1}$, and make $u_{-1}$ incident with a virtual edge
  $e_{-1}$. The construction will proceed along the `path' $e_{-1},
  u_{-1}, e_0, u_0, e_1,\dots, u_r$. The vertex $u_{-1}$ and the edge
  $e_{-1}$ are in no relation to the actual vertex and edge preceding
  $e_0$ (namely, $u_0^-$ and $e_0^-$).
\end{sloppypar}

Let $t$ be the level of $e_0$. We first make a \emph{seed choice} for
the path $P^i$, randomly deciding about the status of the virtual edge
$e_{-1}$ and the virtual vertex $u_{-1}$:
\begin{itemize}
\item with probability $p(t)$, we consider $e_{-1}$ to be in $T$,
\item with probability $q(t)$, we consider $u_{-1}$ to be in $T$,
\item with probability $p(t) = 1 - p(t) - q(t)$, neither of the above
  happens.
\end{itemize}
The choice is independent of the seed choices for the other paths
$P^{i'}$.
  
The rest of the process for the path $P^i$ is deterministic. Let
$j\geq 0$. We specify whether $e_j$ or $u_j$ will be included in $T$,
assuming that the status of $e_s$ and $u_s$ ($s < j$) has been
decided. 

The edge $e_j$ will be added to $T$ if and only if 
\begin{equation}
  \label{eq:incl-edge}
  e_{j-1}\notin T \text{ and } u_{j-1}\notin T.
\end{equation}
(For $j=0$, these events refer to the result of the seed choice.) The
vertex $u_j$ will be included in $T$ if and only if both of the
following hold:
\begin{gather}
  \label{eq:incl-vertex1}
  u_{j-1}\notin T \text{ and } e_j\notin T,\\
  \label{eq:incl-vertex2}
  (u_j^*\notin T \text{ and } \lambda(u_j^*) < \lambda(u_j)) \text{ or
  } \lambda(u_j^*) > \lambda(u_j).
\end{gather}

After all of $P^i$ is processed according to these rules, step $i$ is
completed and if $i<m$, we proceed to the boundary edge $e^{i+1}$.

Once we have completed all $m$ steps, we have obtained the set $T$. It
is not necessarily a total independent set, since the random decision
on $e^i$ and $u^i_0$ did not take into account the real status of the
edge and vertex preceding them in $F$. There can be a similar problem
at the end of the path $P^i$ and, furthermore, the last vertex of
$P^i$ may not be covered by $T$. Before we resolve these problems and
construct the full total independent set $\tilde T$, we analyse the
probability that a given vertex or edge is contained in $T$.

We first derive a lemma concerning the independence of certain
events. Assume that $P$ is a path from $u$ to $v$ in $G$, where
$u,v\in V(G)$. We consider $P$ as directed from $u$ to $v$. We say
that $P$ is \emph{rightward} if $P$ contains no edge in $B$, and the
direction of each edge of $P$ contained in $F$ matches the orientation
of $F$. Given a function $\map \lambda {V(T(G))} {\Setx{1,\dots,k}}$,
we will say that $P$ is \emph{$\lambda$-ascending} if it is rightward
and for every edge $xy$ of $P$ that is not contained in $F$, we have
$\lambda(x) < \lambda(y)$. If there is a $\lambda$-ascending path from
$u$ to $v$, we write $u<_\lambda v$. Since $B$ is $(F,\ell)$-sparse,
the length of a $\lambda$-ascending path is at most $7k\ell + k-1$.

\begin{lemma}
  \label{l:indep-main}
  Let $u,v\in V(G)$ and let $s,t\in\Setx{1,\dots,k}$. Assume that a
  2-factor $F$ and a set $B$ of boundary edges are fixed and that $u$
  and $v$ are not contained in the same component of $F-B$. If $G-u^*$
  contains a rightward $uv$-path $P_{uv}$ of length at most $\ell$,
  then the following hold:
  \begin{enumerate}[\quad(i)]
  \item the events $u\in T$ and $\lambda(v) = s$ are conditionally
    independent provided that $\lambda(u) = t$,
  \item if $s < t$, then the events $u\in T$ and $v\in T$ are
    conditionally independent provided that $\lambda(u) = t$ and
    $\lambda(v) = s$,
  \item if $s < t$, then the events $uu^+\in T$ and $v\in T$ are
    conditionally independent provided that $\lambda(u) = t$ and
    $\lambda(v) = s$ (recall that $u^+$ denotes the successor of $u$
    on $F$).
  \end{enumerate}
\end{lemma}
\begin{proof}
  We start with an important observation. Suppose that, in our
  algorithm, the random choice of a function $\map \lambda {V(T(G))}
  {\Setx{1,\dots,k}}$ has been made. In this situation, we can
  correctly decide whether a vertex $z\in V(G)$ is included in the set
  $T(F,B)$ based on the following information:
  \begin{itemize}
  \item the level $\lambda(w)$ of every vertex $w$ such that $w
    <_\lambda z$ or $w^* <_\lambda z$ (observe that this includes the
    vertex $z^*$), and
  \item the result of the seed choice for every path containing a
    vertex $w$ such that $w <_\lambda z$.
  \end{itemize}

  We now prove (i). Let
  \begin{align*}
    \PP(z) =
    \{&w\in V(G):\,w <_\lambda z \text{ or } w^* <_\lambda z\\
    &\text{for some $\lambda$ such that $\lambda(u)=t$}\}.
  \end{align*}
  We claim that in the component $Q$ of $F-B$ containing $v$, there is
  no vertex $v'$ such that $v'\in\PP(u)$. Suppose the
  contrary. Assuming first that $v'<_\lambda u$ for some $\lambda$, we
  choose a $\lambda$-ascending $v'u$-path $P_{v'u}$ for a suitable
  $\lambda$. Observe that since $P_{v'u}$ does not contain the edge
  following $u$ in $F$ while $P_{uv}$ does, the union $P_{v'u} \cup
  P_{uv} \cup Q$ contains a cycle. Furthermore, the length of the
  cycle is at most $(7k\ell+k-1)+\ell+7\ell$, which is less than
  $15k\ell$ whenever $k,\ell \geq 2$ (which will be the case). This
  contradicts the girth assumption. The proof for the case $(v')^*
  <_\lambda u$ is similar.
  
  Since we can decide about $u\in T$ without the knowledge of
  $\lambda(Q)$, and the choice of $\lambda(Q)$ is independent of all
  the other random choices made during the execution of the algorithm,
  the assertion follows.

  The proofs of (ii) and (iii) are similar; we only prove (ii). For a
  vertex $z$, we define
  \begin{align*}
    \PP'(z) =
    \{&w\in V(G):\,w <_\lambda z \text{ or } w^* <_\lambda z\\
    &\text{for some $\lambda$ such that $\lambda(u)=t$ and $\lambda(v)
      = s$}\}
  \end{align*}
  and note that the knowledge of the levels of vertices in $\PP'(z)$
  and the seed choices for the respective paths suffice for the
  decision whether $z\in T(F,B)$ under the assumption that
  $\lambda(u)=t$ and $\lambda(v) = s$.

  We claim that $\PP'(u) \cap \PP'(v) = \emptyset$. Suppose the
  contrary. Then there exists $w\in V(G)$ such that for suitable
  functions $\mu$ and $\lambda$, the following holds:
  \begin{itemize}
  \item $w<_\mu u$ or $w^*<_\mu u$, and
  \item $w<_\lambda v$ or $w^*<_\lambda v$.
  \end{itemize}
  By symmetry, $w$ may be assumed to be chosen such that $w<_\mu
  u$. Assume further that $w<_\lambda v$. Let $P_{wv}$ be a
  $\lambda$-ascending path from $w$ to $v$. Since $\lambda(u) >
  \lambda(v)$, $u$ is not contained in $P_{wv}$. There is a
  $\mu$-ascending path to $u$ from either $w$ or $w^*$ which
  determines a rightward $wu$-path $P_{wu}$. Unlike $P_{uv}$, this
  path does not contain the edge of $F$ following $u$, so $P_{uv} \cup
  P_{wu} \cup P_{wv}$ contains a cycle, the length of which is at most
  $\ell+1+2(7k\ell+k-1) < 15k\ell$ (whenever $k,\ell \geq 3$, which
  will be the case). This is a contradiction. The case that
  $w^*<_\lambda v$ is similar.

  Since the sets $\PP'(u)$ and $\PP'(v)$ are disjoint, the events
  $u\in T$ and $v\in T$ depend on disjoint sets of independent random
  choices, and they are therefore conditionally independent under the
  assumption that $\lambda(u)=t$ and $\lambda(v)=s$.
\end{proof}

In the proof of Lemma~\ref{l:v-e-prob} below, we will need a standard
fact on conditional probability (which is easily verified by direct
computation): 

\begin{lemma}[Rule of contraction for conditional probability]
  \label{l:contraction}
  Let $A$, $B$, $C$, $D$ be random events. Assume that:
  \begin{enumerate}[\quad(1)]
  \item $A$ and $B$ are conditionally independent given $C \wedge D$,
    and
  \item $A$ and $C$ are conditionally independent given $D$.
  \end{enumerate}
  Then $A$ is conditionally independent of $B \wedge C$ given $D$.
\end{lemma}

The following lemma is a fundamental observation on the behaviour of
the algorithm described in this section.

\begin{lemma}\label{l:v-e-prob}
  Let $u\in V(G)$ and $e\in E(F)$ and let
  $t\in\Setx{1,\dots,k}$. Then:
  \begin{enumerate}[\quad(i)]
  \item $\cprob{e\in T}{\lambda(e) = t} = p(t)$,
  \item $\cprob{u\in T}{\lambda(u) = t \wedge \lambda(u^*) > t} = p(t)$,
  \item $\cprob{u\in T}{\lambda(u) = t \wedge \lambda(u^*) = t} = 0$,
  \item $\cprob{u\in T}{\lambda(u) = t} = q(t)$,
  \end{enumerate}
\end{lemma}

\begin{proof}
  Let $u = u^i_j$ and $e = e^i_j$ in the notation introduced above. We
  prove all the claims simultaneously by double induction on $t$ and
  $j$: we show that if the claims hold for every vertex $u^{i'}_{j'}$
  and edge $e^{i'}_{j'}$ whose level is $t'$, such that $(t',j')$
  precedes $(t,j)$ in the lexicographic order, then they also hold for
  $u$ and $e$. The base case $t=1$ and $j=-1$ (virtual vertex or edge)
  follows directly from the construction.

  Consider assertion (i). By the rule for the inclusion of an edge in
  $T$, $e\in T$ if and only if neither $e^-\in T$ nor $u^-\in T$. By
  the induction hypothesis, the latter two events occur with
  probability $p(t)$ and $q(t)$, respectively. (All the probabilities
  in this proof are relative to the condition $\lambda(u) = t$.) Since
  the events are disjoint, the probability that none occurs is
  $1-p(t)-q(t) = p(t)$ as claimed.

  The proof of (ii) is similar: given the assumption that
  $\lambda(u^*) > t$, the condition (\ref{eq:incl-vertex2}) for the
  inclusion of $u$ (on page~\pageref{eq:incl-vertex2}) is vacuously
  true. Thus $u$ is included if and only if condition
  (\ref{eq:incl-vertex1}) holds, which happens with probability
  $1-p(t)-q(t) = p(t)$.

  Part (iii) is clear since $u$ is never added to $T$ if $\lambda(u) =
  \lambda(u^*)$. 

  It remains to prove (iv). Here we know that (\ref{eq:incl-vertex1})
  again holds with probability $p(t)$. To assess the probability of
  (\ref{eq:incl-vertex2}), let us compute
  \begin{align*}
    &\cprob{(u_j^*\notin T \wedge \lambda(u_j^*) < \lambda(u_j))
      \vee (\lambda(u_j^*) > \lambda(u_j))}{\lambda(u) = t}\\
    &= \sum_{i=1}^{t-1}\cprob{u_j^*\notin T \wedge \lambda(u_j^*) =
      i}{\lambda(u) = t} + \sum_{i = t+1}^k\cprob{\lambda(u_j^*) =
      i}{\lambda(u) = t}\\
    &= \sum_{i=1}^{t-1} \frac{1-p(i)}k + \sum_{i=t+1}^k\frac1k,
  \end{align*}
  where the last equality follows from the induction hypothesis. Since
  $q(t)$ is just the product of the result with $p(t)$, we need to
  show that (\ref{eq:incl-vertex1}) and (\ref{eq:incl-vertex2}) are
  conditionally independent given the condition $\lambda(u) = t$. To
  rephrase this task, let us write
  \begin{align*}
    X_1 &\equiv u^-\in T, &
    Y_1 &\equiv u^*\notin T \text{ and } \lambda(u^*) < t,\\
    X_2 &\equiv e^-\in T, &
    Y_2 &\equiv \lambda(u^*) > t,
  \end{align*}
  so that (\ref{eq:incl-vertex1}) is equivalent to $\overline{X_1\vee
    X_2}$ and (\ref{eq:incl-vertex2}) is equivalent to $Y_1\vee Y_2$
  (assuming $\lambda(u) = t$). By basic facts of probability, the
  above conditional independence will be established if we can show
  that each $X_i$ is conditionally independent of each $Y_j$
  ($i,j\in\Setx{1,2}$) given that $\lambda(u) = t$.

  For $j=2$, this follows directly from Lemma~\ref{l:indep-main} (i)
  by expressing $Y_2$ as the union of disjoint events
  $\Set{\lambda(u^*)=i}{i=t+1,\dots,k}$. For $j=1$, we apply
  Lemma~\ref{l:contraction}, substituting $X_i$ for $A$ (where
  $i=1,2$), $u^*\notin T$ for $B$, $\lambda(u^*) < t$ for $C$ and
  $\lambda(u) = t$ for $D$. The hypothesis of the lemma is satisfied
  by Lemma~\ref{l:indep-main}~(ii) and (iii), so it follows that $X_i$
  and $Y_1$ are conditionally independent given $\lambda(u) = t$ as
  required. The proof is complete.
\end{proof}

Lemma~\ref{l:v-e-prob} enables us to compute the probability that any
vertex of $G$ or edge of $F$ is in $T$. Note that the probabilities do not
depend on $G$:
\begin{observation}
  \label{obs:v-e-average}
  Let $v\in V(G)$ and $e\in E(F)$. Then
  \begin{align*}
    \prob{e\in T} &= p^* := \sum_{i=1}^k \frac{p(i)}k,\\
    \prob{v\in T} &= q^* := \sum_{i=1}^k \frac{q(i)}k.
  \end{align*}
\end{observation}
\begin{proof}
  The assertions follow from Lemma~\ref{l:v-e-prob}(i) and (iv).
\end{proof}

Phase 1 of the construction is now complete. Let us summarize: we have
constructed a set $T$ whose restriction to any component of $F-B$ is
total independent. The inclusion of a vertex of $G$ in $T$ is the same
for all vertices (even when conditioned on the level of the
vertex). The same holds for the inclusion of an edge of $F$ in $T$.

\begin{phase}
  We modify $T$ to a full total independent set $\tilde T$.
\end{phase}

Let us examine the possible reasons why $T$ is not full and total
independent in detail. Consider a boundary edge $e_0=e^i$ and its end
$u_0=u_0^i$ in $P^i$, and suppose that $e^i$ is also incident with a
path $P^j$. Let $u'$ denote the last vertex of $P^j$ (thus, $u' =
u_0^-$) and write $u'' = (u')^-$ and $e' = u''u'$. Recall that if $u$
is a vertex of $G$, then $u^*$ denotes its mate.

A \emph{conflict} at $e^i$ is any of the situations listed in the
middle column of Table~\ref{tab:conflicts}; the right hand column
shows how to modify $T$ in order to resolve the conflict. Note that
all the cases are mutually exclusive and that the resolution rules are
deterministic. The conflict types are shown in
Figure~\ref{fig:conflicts}.

\begin{table}
  \begin{center}
    \renewcommand{\arraystretch}{1.2}
    \begin{tabular}{c|p{.45\textwidth}|l}
      type & situation & action on $T$\\\hline
      I & $u'\in T$ and $u_0\in T$ & 
      replace $u'$ and $u_0$ by $e_0$\\ 
      II & $u'\in T$ and $e_0\in T$ & 
      remove $u'$\\
      IIIa & $\Setx{e',e_0,u_1}\subseteq T$ &
      replace $e_0$ and $u_1$ by $e_1$\\
      IIIb & $\Setx{e',e_0,u_0^*}\subseteq T$ and $u_1\notin T$ &
      replace $e_0$ and $u_0^*$ by $u_0u_0^*$\\
      IIIc & $\Setx{e',e_0}\subseteq T$ and $u_1,u_0^*\notin T$ &
      replace $e_0$ by $u_0$\\
      IVa & $u'$ is not covered by $T$, $u_0\in T$ &
      replace $u_0$ by $e_0$\\
      IVb & $u'$ is not covered by $T$, $u_0\notin T$, $u''\in T$ &
      replace $u''$ by $e'$\\
      IVc & $u'$ is not covered by $T$, $u_0,u''\notin T$,
      $(u')^*\in T$ &
      replace $(u')^*$ by $u'(u')^*$
    \end{tabular}
  \end{center}
  \caption{The types of conflicts.}
  \label{tab:conflicts}
\end{table}

\setlength{\tabcolsep}{.28cm}
\newlength{\captionspace}
\setlength{\captionspace}{.8cm}
\newcommand{\virtualleft}{%
  \path (-2,0) %
    node[vertex,white] {}%
    node[lbl,white] {$u''$};%
  }

\begin{sloppypar}
\renewcommand{\arraystretch}{3.5}
\renewcommand{\tabularxcolumn}[1]{>{\arraybackslash}m{#1}}
\newcommand{\cnum}[1]{\begin{minipage}{1cm}\raggedleft#1\vspace{7mm}\end{minipage}}
\begin{figure}
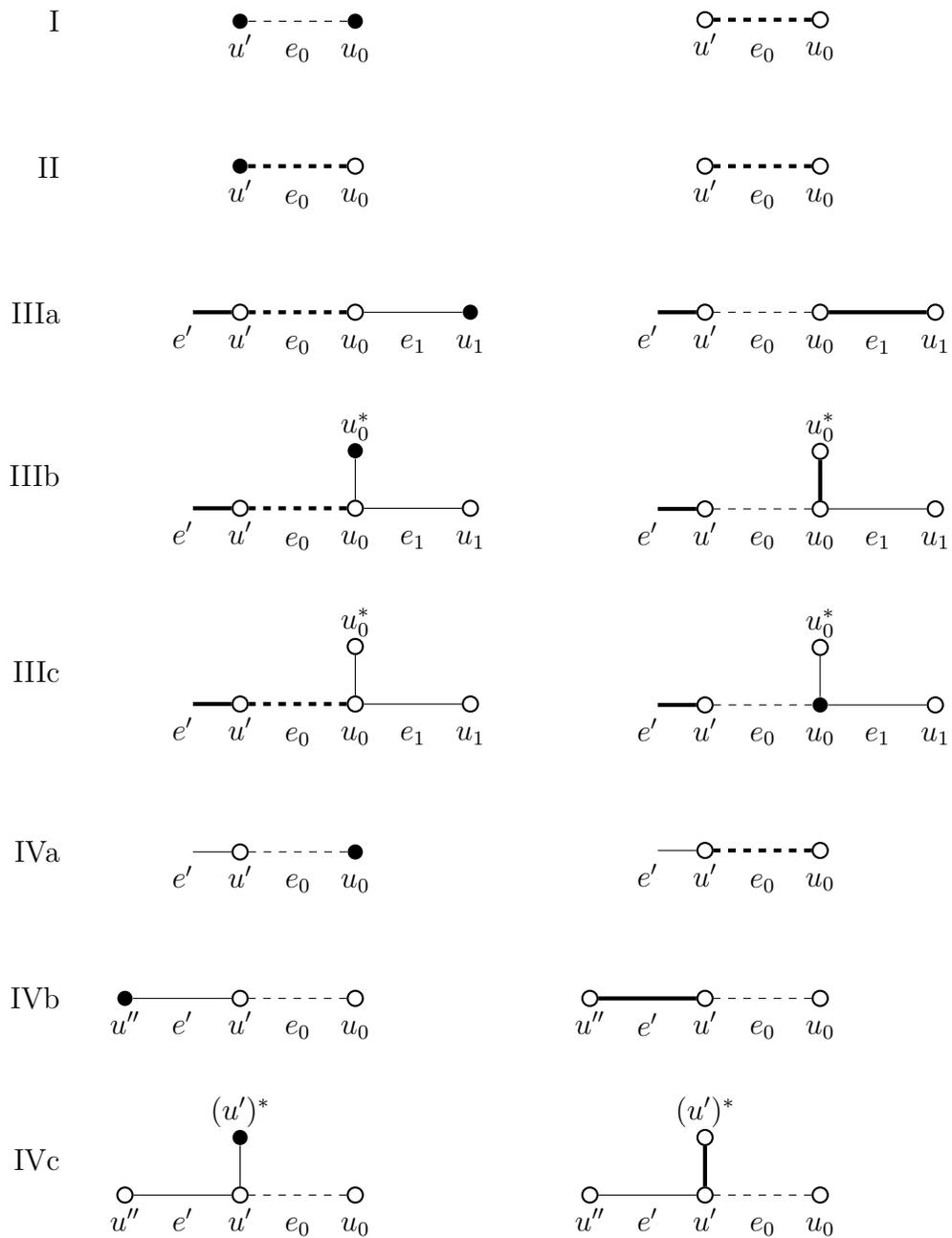

  \begin{center}
    \begin{tabularx}{\textwidth}{>{\raggedleft}p{1cm}XX}
      \cnum{I} & 
      \begin{conflict}
        \virtualleft
        \path (0,0) node[vertex] (a) {} node[lbl] {$u'$};
        \path (2,0) 
          node[lbl] {$u_0$} 
          node[vertex] (b) {} 
          edge[thin,dashed] node[lbl] {$e_0$} (a);
      \end{conflict} & 
      \begin{conflict}
        \virtualleft
        \path (0,0) node[whitev] (a) {} node[lbl] {$u'$};
        \path (2,0) 
          node[lbl] {$u_0$} 
          node[whitev] (b) {} 
          edge[ultra thick,dashed] node[lbl] {$e_0$} (a);
      \end{conflict} \\
    \cnum{II} & 
    \begin{conflict}
      \virtualleft
      \path (0,0) node[vertex] (a) {} node[lbl] {$u'$};
      \path (2,0) node[lbl] {$u_0$} node[whitev] (b) {} 
          edge[ultra thick,dashed] node[lbl] {$e_0$} (a);
    \end{conflict} &
    \begin{conflict}
      \virtualleft
      \path (0,0) node[whitev] (a) {} node[lbl] {$u'$};
      \path (2,0) node[lbl] {$u_0$} node[whitev] (b) {} 
          edge[ultra thick,dashed] node[lbl] {$e_0$} (a);
    \end{conflict} \\
    \cnum{IIIa} & 
    \begin{conflict}
      \virtualleft
      \path (-1,0) node (c) {}
        node[lbl] {$e'$};      
      \path (0,0) 
        node[whitev] (a) {} 
        edge[ultra thick] (c)
        (a) node[lbl] {$u'$};
      \path (2,0) node[lbl] {$u_0$} 
        node[whitev] (b) {} 
        edge[ultra thick,dashed] node[lbl] {$e_0$} (a);
      \path (4,0) 
        node[lbl] {$u_1$}
        node[vertex] (d) {} 
        edge[thin] node[lbl] {$e_1$}
        (b);
    \end{conflict} &
    \begin{conflict}
      \virtualleft
      \path (-1,0) node (c) {}
        node[lbl] {$e'$};      
      \path (0,0) 
        node[whitev] (a) {} 
        edge[ultra thick] (c)
        (a) node[lbl] {$u'$};
      \path (2,0) node[lbl] {$u_0$} 
        node[whitev] (b) {} 
        edge[thin,dashed] node[lbl] {$e_0$} (a);
      \path (4,0) 
        node[lbl] {$u_1$}
        node[whitev] (d) {} 
        edge[ultra thick] node[lbl] {$e_1$}
        (b);
    \end{conflict}\\
    \cnum{IIIb} & 
    \begin{conflict}
      \virtualleft
      \path (-1,0) node (c) {}
        node[lbl] {$e'$};            
      \path (0,0) node[whitev] (a) {} 
        edge[ultra thick] (c)
        (a) node[lbl] {$u'$};
      \path (2,0) node[lbl] {$u_0$}
        node[whitev] (u0) {} 
        edge[ultra thick,dashed] node[lbl] {$e_0$} (a);
      \path (2,1) 
        node[above] {$u_0^*$}        
        node[vertex] (u0s) {} 
        edge[thin] (u0);
      \path (4,0)
        node[lbl] {$u_1$}
        node[whitev] (d) {} 
        edge[thin] node[lbl] {$e_1$}
        (u0);
    \end{conflict} &
    \begin{conflict}
      \virtualleft
      \path (-1,0) node (c) {}
        node[lbl] {$e'$};            
      \path (0,0) node[whitev] (a) {} 
        edge[ultra thick] (c)
        (a) node[lbl] {$u'$};
      \path (2,0) node[lbl] {$u_0$}
        node[whitev] (u0) {} 
        edge[thin,dashed] node[lbl] {$e_0$} (a);
      \path (2,1) 
        node[above] {$u_0^*$}        
        node[whitev] (u0s) {} 
        edge[ultra thick] (u0);
      \path (4,0)
        node[lbl] {$u_1$}
        node[whitev] (d) {} 
        edge[thin] node[lbl] {$e_1$}
        (u0);
    \end{conflict} \\
    \cnum{IIIc} & 
    \begin{conflict}
      \virtualleft
      \path (-1,0) node (c) {}
        node[lbl] {$e'$};      
      \path (0,0) node[whitev] (a) {} 
        edge[ultra thick] (c)
        (a) node[lbl] {$u'$};
      \path (2,0) node[lbl] {$u_0$}
        node[whitev] (u0) {} 
        edge[ultra thick,dashed] node[lbl] {$e_0$} (a);
      \path (2,1) 
        node[above] {$u_0^*$}        
        node[whitev] (u0s) {} 
        edge[thin] (u0);
      \path (4,0) 
        node[lbl] {$u_1$}
        node[whitev] (d) {} 
        edge[thin] node[lbl] {$e_1$}
        (u0);
    \end{conflict} &
    \begin{conflict}
      \virtualleft
      \path (-1,0) node (c) {}
        node[lbl] {$e'$};      
      \path (0,0) node[whitev] (a) {} 
        edge[ultra thick] (c)
        (a) node[lbl] {$u'$};
      \path (2,0) node[lbl] {$u_0$}
        node[vertex] (u0) {} 
        edge[thin,dashed] node[lbl] {$e_0$} (a);
      \path (2,1) 
        node[above] {$u_0^*$}        
        node[whitev] (u0s) {} 
        edge[thin] (u0);
      \path (4,0) 
        node[lbl] {$u_1$}
        node[whitev] (d) {} 
        edge[thin] node[lbl] {$e_1$}
        (u0);
    \end{conflict} \\
    \cnum{IVa} & 
    \begin{conflict}
      \virtualleft
      \path (-1,0) node (c) {}
        node[lbl] {$e'$};            
      \path (0,0) node[whitev] (um1) {} edge[thin] (c) 
        (um1) node[lbl] {$u'$};
      \path (2,0) node[lbl] {$u_0$}
          node[vertex] (u0) {} 
          edge[thin,dashed] node[lbl] {$e_0$} (a);
    \end{conflict} &
    \begin{conflict}
      \virtualleft
      \path (-1,0) node (c) {}
        node[lbl] {$e'$};            
      \path (0,0) node[whitev] (um1) {} edge[thin] (c) 
        (um1) node[lbl] {$u'$};
      \path (2,0) node[lbl] {$u_0$}
          node[whitev] (u0) {} 
          edge[ultra thick,dashed] node[lbl] {$e_0$} (a);
    \end{conflict}\\
    \cnum{IVb} & 
    \begin{conflict}
      \path (-2,0) 
        node[vertex] (um2) {}
        node[lbl] {$u''$}; 
      \path (0,0) 
        node[whitev] (um1) {} 
        edge[thin] node[lbl] {$e'$} 
          (um2)
        (um1) node[lbl] {$u'$};
      \path (2,0) 
         node[lbl] {$u_0$}
         node[whitev] (u0) {} 
         edge[thin,dashed] node[lbl] {$e_0$} (a);
    \end{conflict} &
    \begin{conflict}
      \path (-2,0) 
        node[whitev] (um2) {}
        node[lbl] {$u''$}; 
      \path (0,0) 
        node[whitev] (um1) {} 
        edge[ultra thick] node[lbl] {$e'$} 
          (um2)
        (um1) node[lbl] {$u'$};
      \path (2,0) 
         node[lbl] {$u_0$}
         node[whitev] (u0) {} 
         edge[thin,dashed] node[lbl] {$e_0$} (a);
    \end{conflict} \\
    \cnum{IVc} & 
    \begin{conflict}
      \path (-2,0) 
        node[whitev] (um2) {}
        node[lbl] {$u''$}; 
      \path (0,0) 
        node[whitev] (um1) {} 
        edge[thin] 
          node[lbl] {$e'$} 
          (um2)
        (um1) node[lbl] {$u'$};
      \path (2,0) 
        node[lbl] {$u_0$}
        node[whitev] (u0) {} 
        edge[thin,dashed] node[lbl] {$e_0$} (a);
      \path (0,1)
        node[above] {$(u')^*$}
        node[vertex] (um1s) {} 
        edge[thin] (um1);
    \end{conflict} &
    \begin{conflict}
      \path (-2,0) 
        node[whitev] (um2) {}
        node[lbl] {$u''$}; 
      \path (0,0) 
        node[whitev] (um1) {} 
        edge[thin] 
          node[lbl] {$e'$} 
          (um2)
        (um1) node[lbl] {$u'$};
      \path (2,0) 
        node[lbl] {$u_0$}
        node[whitev] (u0) {} 
        edge[thin,dashed] node[lbl] {$e_0$} (a);
      \path (0,1)
        node[above] {$(u')^*$}
        node[whitev] (um1s) {} 
        edge[ultra thick] (um1);
    \end{conflict}
  \end{tabularx}
  \end{center}
  \caption{Possible conflict types. The figures on the left are the
    conflict situations, those on the right show the resolution of the
    conflict. The boundary edge is shown dashed; thick edges and black
    vertices are those included in $T$.}
  \label{fig:conflicts}
\end{figure}
\end{sloppypar}

The resolution of a conflict at $e^i$ only affects vertices and edges
in $N_2(e^i)$. Since $B$ is 4-distant, each vertex and edge is in at most
one set $N_2(e^i)$. Thus, the order in which the conflicts are
resolved is irrelevant. It is easy to see that after the resolution of
all the conflicts, the resulting set $\tilde T$ is total independent
and full.

We need to show that the conflicts occur in a uniform manner
throughout $G$, i.e., that if $e,f\in B$, then the probability of a
conflict of any given type is the same at $e$ and $f$. As an example,
we consider the conflict type IIIb and sketch how to prove this claim.

Observe that, under the assumption $\lambda(e_0)=t$, the conflict of
type IIIb occurs if and only if $X_1 \wedge X_2 \wedge X_3 \wedge X_4$
occurs, where:
\begin{align*}
  X_1 &\ \equiv\ e'\in T,\\
  X_2 &\ \equiv\ e_0\in T,\\
  X_3 &\ \equiv\ (u_0^*)^-\notin T \ \wedge\ \bigl((u_0^*)^-u_0^*\notin T \wedge
    \lambda(u_0^*) \neq t\bigr),\\
  X_4 &\ \equiv\ (\lambda(u_1^*) < t \wedge u_1^*\in T) \ \vee\ \lambda(u_1^*)=t.
\end{align*}
The conditional probability of each of these events (with respect to
$\lambda(e_0)=t$) is not hard to compute using the results proved
earlier in this section. Reasoning similarly as in the proof of
Lemma~\ref{l:indep-main}, one can show that each set of events
$\Setx{X_1,X_2,Y_3,Y_4}$ is conditionally mutually independent given
$\lambda(e_0)=t$, where $Y_3$ ($Y_4$) ranges over the `summands' in
the disjunction $\overline{X_3}$ ($X_4$, respectively). From this, it
is a simple exercise in the use of Lemma~\ref{l:contraction} to
conclude that the events $X_i$ are conditionally mutually independent
given that $\lambda(e_0)=t$. In particular, the probability
$P_{\mathrm{IIIb}}$ of the conflict of type IIIb is
\begin{align*}
  P_{\mathrm{IIIb}} &= \sum_{t=1}^k \cprob{X_1}{\lambda(e_0)=t} \cdot 
  \cprob{X_2}{\lambda(e_0)=t} \cdot \\
  &\qquad\cprob{X_3}{\lambda(e_0)=t} \cdot 
  \cprob{X_4}{\lambda(e_0)=t}\\
  &= (p^*)\cdot p(t) \cdot \left( p^*-\frac{p(t)}k \right)
  \cdot \frac1k\left(1+\sum_{i=1}^{t-1} p(i)\right),
\end{align*}
where the last equality follows from Lemma~\ref{l:v-e-prob} and
Observation~\ref{obs:v-e-average} (we point out the use of
Lemma~\ref{l:v-e-prob}(ii) to compute the conditional probability of
$X_4$). Note that the resulting value of $P_{\mathrm{IIIb}}$ does not
depend on $e_0$. For all of the other conflict types, a similar
computation applies.

The way we resolve conflicts of type IIIb decreases the probability
that $e_0 \in \tilde T$ by $P_{\mathrm{IIIb}}$ when comparing to
$\prob{e_0\in T}$. The final probability $\prob{e_0\in \tilde T}$ can
be determined by considering all the conflict types whose resolution
involves $e_0$, namely types I, IIIa, IIIb, IIIc and IVa. It is
important that $\prob{x\in \tilde T}$ only depends on the position of
$x\in V(T(G))$ relative to $B$. To formalize this notion, let us
define the \emph{$(F,B)$-type} (or just \emph{type}) of $x$ as
follows.

\begin{figure}
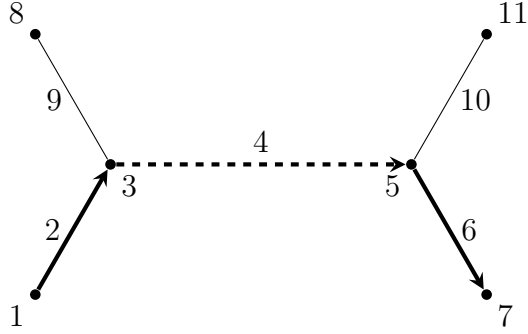

  \begin{center}
    \begin{tikzgraph}[scale=2]
      \path (0,0) node[vertex] (a) {} node[below right] {3};
      \path ($(a) + (120:1)$) node[vertex] (a1) {} node[above left]{8};
      \path ($(a) + (240:1)$) node[vertex] (a2) {} node[below left]{1};
      \path (2,0) node[vertex] (b) {} node[below left] {5};
      \path ($(b) + (60:1)$) node[vertex] (b1) {} node[above right]{11};
      \path ($(b) + (-60:1)$) node[vertex] (b2) {} node[below right]{7};
      \path (a) edge[thin] node[left] {9} (a1) 
        edge[ultra thick=4pt,<-] node [left] {2} (a2);
      \path (a) edge[ultra thick=4pt,dashed,->] node[above] {4} (b);
      \path (b) edge[thin] node[right]{10} (b1) 
        edge[ultra thick=4pt,->] node[right]{6} (b2);
    \end{tikzgraph}
  \end{center}
  \caption{The definition of an $(F,B)$-type.}
  \label{fig:type}
\end{figure}

Assume that $x\in N_2(e)$, where $e\in B$. Consider the graph $H$ in
Figure~\ref{fig:type} and the unique isomorphism between $N_2(e)$
(viewed as a subgraph of $G$) and $H$, taking the edges of $F$ to the
bold edges in such a way that their orientations match. We define the
$(F,B)$-type of $x$ as the label associated to the image of $x$ in
$H$. Thus, the type is an integer from $\Setx{1,\dots,11}$. Note that
it is only defined for the vertices and edges of $N_2(B)$. 

\begin{observation}
  \label{obs:probs-main}
  For $x,y\in V(T(G))$, the following holds:
  \begin{enumerate}[\quad(i)]
  \item if $x,y\in V(G)-N_2(B)$, then 
    \begin{equation*}
      \prob{x\in\tilde T} = \prob{y\in\tilde T} = \prob{x\in T},
    \end{equation*}
  \item if $x,y\in E(F)-N_2(B)$, then 
    \begin{equation*}
      \prob{x\in\tilde T} = \prob{y\in\tilde T} = \prob{x\in T},
    \end{equation*}
  \item if $x,y\in N_2(B)$ and the $(F,B)$-type of $x$ and $y$ is the
    same, then
    \begin{equation*}
      \prob{x\in\tilde T} = \prob{y\in\tilde T}.
  \end{equation*}
  \end{enumerate}
\end{observation}


\section{Cubic bridgeless graphs}
\label{sec:cubic}

In this section, we prove Theorem~\ref{t:main} under the assumption
that $G$ is a cubic bridgeless graph, deferring the general case to
Section~\ref{sec:subcubic}. Recall our assumption that the girth of
$G$ is at least $15k\ell$, where $k$ and $\ell$ are appropriately
chosen constants to be determined in Lemma~\ref{l:weights2}. Let $F$
be a 2-factor in $G$ and let $B\subseteq E(F)$ be an $(F,\ell)$-sparse
set of edges.

At this point, we need to introduce the following concept and
result. A graph $H$ is \emph{strongly $r$-colourable} if for any
partition of $V(H)$ into $\lceil\size{V(H)}/r\rceil$ parts, each of
size at most $r$, $H$ admits a proper $r$-colouring with each colour
class intersecting each part of the partition in at most one
vertex. It is known~\cite{bib:Fel-transversals} that a strongly
$r$-colourable graph is also strongly $(r+1)$-colourable. It is
therefore natural to define the \emph{strong chromatic number} of $H$
as the smallest $r$ such that $H$ is strongly
$r$-colourable. Haxell~\cite{bib:Hax-strong} proved the following
upper bound on the strong chromatic number, improving an earlier
result of Alon~\cite{bib:Alo-strong} (see
also~\cite{bib:Hax-improved}).

\begin{theorem}\label{t:strong}
  The strong chromatic number of $H$ is at most $3\Delta(H)-1$.
\end{theorem}

We will use Theorem~\ref{t:strong} to show that under certain
conditions, $E(F)$ can be decomposed into $(F,\ell)$-sparse sets. In
the following result, all that we need in this section is the special
case $E(Q)=\emptyset$; the general statement will be used in Section
\ref{sec:subcubic}.

\begin{lemma}\label{l:decompose}
  Let $F$ be a 2-factor of $G$ and let $Q$ be a graph with vertex set
  $E(F)$. If $\ell \geq 83 + 3\Delta(Q)$, then the set $E(F)$ can be
  decomposed into $3\ell$ sets, each of which is $(F,\ell)$-sparse and
  none of which contains a pair of edges that forms an edge of $Q$.
\end{lemma}
\begin{proof}
  We first use Theorem~\ref{t:strong} to partition $E(F)$ into
  4-distant sets containing no pairs of edges which form an edge of
  $Q$. Consider an auxiliary graph $H$ with vertex set $E(F)$ and an
  edge $ef$ for each pair $e,f\in E(F)$ such that either the distance
  between $e$ and $f$ in $G$ is at most 3, or $ef\in E(Q)$. It is easy
  to see that the maximum degree of $H$ is at most $28+\Delta(Q)$.

  Let the cycles of $F$ be $C_1,\dots,C_n$. For $1\leq i \leq n$, set
  $m(i) = \left\lceil\size{C_i}/\ell\right\rceil$. Split $C_i$ into
  edge-disjoint paths $P_{i,1},\dots,P_{i,m(i)}$ such that each
  $P_{i,j}$ with $j\geq 2$ has length $\ell$, while $P_{i,1}$ has
  length at least one. Let $\PP$ be a partition of $E(F)$ such that
  the edge set of each path $P_{i,j}$, where $1\leq i\leq s$ and
  $j\geq 2$, forms a class of $\PP$, and moreover all but at most one
  class of $\PP$ are of size $\ell$.

  Since $\ell\geq 83+3\Delta(Q)\geq 3\Delta(H)-1$,
  Theorem~\ref{t:strong} (applied to $\PP$) implies that there is a
  colouring (say, $c$) of the edges of $F$ by $\ell$ colours such that
  each colour class $B_1,\dots,B_{\ell}$ intersects each set in $\PP$
  in at most one edge. It follows that each $B_r$ contains exactly one
  edge from each $P\in\PP$ with $\size P=\ell$. Furthermore, by the
  construction of $H$, each $B_r$ is 4-distant and no $B_r$ contains
  edges $e,f$ with $ef\in E(Q)$.

  We now construct the desired partition of $E(F)$ into
  $(F,\ell)$-sparse sets $B_{r,t}$, where $r\in\Setx{1,\dots,\ell}$
  and $t\in\Setx{0,1,2}$. Each $B_{r,t}$ will be a subset of $B_r$. By
  the definition of an $(F,\ell)$-sparse set, all we need to ensure is
  that each component of $F-B_{r,t}$ is a path of length between
  $\ell$ and $7\ell$.

  For $i=1,\dots,n$, we construct a sequence $s_{i,1}, \dots,
  s_{i,m(i)}$ of symbols $0,1,2$ starting with $01$, ending with $2$,
  and such that every two consecutive occurences of the same symbol
  are separated by one, two or three other symbols (when considering
  the first and last symbol as adjacent). Since the girth of $G$ is at
  least $15k\ell > 6\ell$, it suffices to construct the sequence for
  each length starting with 6. We start with one of the following
  sequences depending on the residue class mod 3 of $m(i)$:
  \begin{align*}
    012012 & \text{\qquad for $m(i) \equiv 0 \pmod 3$,}\\
    0102012 & \text{\qquad for $m(i) \equiv 1 \pmod 3$,}\\
    01021012 & \text{\qquad for $m(i) \equiv 2 \pmod 3$}
  \end{align*}
  and insert a suitable number of blocks of 201 before the last symbol
  2 to make the length equal to $m(i)$.
 
  The set $B_{r,t}$ ($t\in\Setx{0,1,2}$) is defined as the
  intersection of $B_r$ with the edge sets of all paths $P_{i,j}$ such
  that $s_{i,j} = t$, where $i=1,\dots,n$ and $j=1,\dots,m(i)$. For
  $i=1,\dots,n$, each symbol of the sequence $s_{i,1}\dots
  s_{i,m(i)}$, except possibly for $s_{i,1}$, represents $\ell$
  consecutive edges. Furthermore, any two neighbouring symbols in this
  sequence as well as the second and last symbol are different. It
  follows that the distance on $C_i$ between any two edges in $B_{r,t}
  \cap E(C_i)$ is at least $\ell$.

  The upper bound follows from the fact that neighbouring occurences
  of any symbol $t\in\Setx{0,1,2}$ are separated by at most three
  other symbols. For $t=0$, this can be improved: neighboring
  occurences of 0 are separated by at most two symbols. At the same
  time, all but at most one symbol in the sequence correspond to paths
  containing edges of all $\ell$ colours. An easy case analysis
  implies that the components of $C_i-B_{r,t}$ ($t\in\Setx{0,1,2}$)
  are paths of length at most $7\ell$. Thus, the sets $B_{r,t}$ are
  indeed $(F,\ell)$-sparse.
\end{proof}

Recall the values $p^*$ and $q^*$, defined in
Observation~\ref{obs:v-e-average}. The following lemma summarizes the
findings of Section~\ref{sec:algorithm}:

\begin{lemma}
  \label{l:weights}
  If $F$ is an oriented 2-factor of $G$ and $B$ is a 4-distant set of
  edges, then there exists a function $\map {w_{F,B}}{\Phi(G)}
  {[0,1]}$ satisfying, for all $x\in V(T(G))$, the following
  conditions:
  \begin{enumerate}[\quad(i)]
  \item if $x\notin N_2(B)$, then 
    \begin{equation*}
      w_{F,B}[x] =
      \begin{cases}
        q^* & \text{for $x\in V(G)$,}\\
        p^* & \text{for $x\in E(F)$,}\\
        0 & \text{for $x\in E(G) - E(F)$,}
      \end{cases}
    \end{equation*}
  \item if $x_1,x_2 \in N_2(B)$ have the same
    $(F,B)$-type, then $w_{F,B}[x_1] = w_{F,B}[x_2]$.
  \end{enumerate}
\end{lemma}
\begin{proof}
  Consider the algorithm, described in Section~\ref{sec:algorithm},
  that produces a total independent set $\tilde T(F,B)$. For any full
  total independent set $X$ in $G$, let $w_{F,B}(X)$ be the
  probability that $\tilde T(F,B) = X$.

  If $v\in V(G) - N_2(B)$, then
  \begin{equation*}
    w_{F,B}[v] = \sum_{v\in X\in\Phi(G)} w_{F,B}(X) = \prob{v\in\tilde T(F,B)}
    = \prob{v\in T(F,B)} = q^*
  \end{equation*}
  by Observations~\ref{obs:v-e-average} and
  \ref{obs:probs-main}(i). The rest of part (i) is derived
  similarly. Part (ii) follows from
  Observation~\ref{obs:probs-main}(iii).
\end{proof}

By combining Lemmas~\ref{l:decompose} and \ref{l:weights}, we obtain
the following corollary which (unlike Lemma~\ref{l:weights}) is no
longer related to a particular edge set $B$:

\begin{lemma}
  \label{l:weights2}
  There are positive rational constants $\alpha$, $\beta$ and $\gamma$
  such that $\alpha+\beta+2\gamma=1$, $\beta > 1/4$, $\alpha \leq
  4/(3\ell)$ and the following holds: If $F$ is a 2-factor of $G$,
  then there exists a function $\map w {\Phi(G)} {[0,1]}$ such that:
  \begin{equation*}
    w[x] =
    \begin{cases}
      \beta & \text{if $x\in V(G)$},\\
      \gamma & \text{if $x\in E(F)$,}\\
      \alpha & \text{if $x\in E(G)-E(F)$.}
    \end{cases}
  \end{equation*}
  for all $x\in V(T(G))$.
\end{lemma}
\begin{proof}
  In this proof, we determine the requirements on the constants $k$
  and $\ell$.  We use Lemma~\ref{l:decompose} (with $E(Q)=\emptyset$)
  to find a decomposition $\BB$ of $E(F)$ into $3\ell$
  $(F,\ell)$-sparse sets; this can be done whenever $\ell\geq 83$ but
  we will require $\ell\geq 96$ to be consistent with the rest of the
  proof. For each $(F,\ell)$-sparse set $B\in\BB$, consider the
  function $w_{F,B}$ of Lemma~\ref{l:weights}, and define
  \begin{equation*}
    w = \sum_{B\in\BB} \frac{w_{F,B}}{3\ell}.
  \end{equation*}
  Since each edge of $F$ is contained in exactly one $B\in\BB$, the
  number of times that $x\in V(T(G))$ acquires a particular
  $(F,B)$-type as $B$ ranges over $\BB$ is independent of $x$. It
  follows that the change in probabilities associated with the
  resolution of conflicts is the same for all $x\in V(G)$, for all
  $x\in E(F)$ and for all $x\in E(G)-E(F)$. In this way, the values
  $0, q^*$ and $p^*$ of Lemma~\ref{l:weights} change into $\alpha$,
  $\beta$ and $\gamma$, respectively.

  We claim that for large $\ell$, $\beta$ is close to $q^*$. To see
  this, observe that every vertex $v\in V(G)$ is contained in exactly
  six of the sets $N_2(e)$, where $e\in E(F)$. By
  Lemma~\ref{l:weights},
  \begin{equation*}
    \size{\beta - q^*} \leq \frac6{3\ell}.
  \end{equation*}
  Furthermore, as $k$ grows large, $q^*$ tends to $1-2(3-\sqrt{7})
  \approx 0.2915$ by Lemma~\ref{l:sqrt}. Thus, for large enough $k$
  and $\ell$, we will have $\beta > 1/4$. In fact, it is routine to
  check that, for instance, the values $k=11$ and $\ell=96$ are
  sufficient.

  It remains to prove that $\alpha \leq 4/(3\ell)$. For any particular
  choice of $B$, an edge $e$ of $E(G)-E(F)$ may only be included in
  $\tilde T$ if it is incident with an edge of $B$. Since this will
  happen for 4 out of the $3\ell$ choices for $B$, the inequality
  follows.
\end{proof}

We can now prove Theorem~\ref{t:main} for cubic bridgeless graphs. By
Theorem~\ref{t:edge}, such a graph $G$ has a fractional
3-edge-colouring $c$. This is equivalent to the existence of perfect
matchings $M_1,\dots,M_{3N}$ such that each edge is contained in
exactly $N$ of them. For $1\leq i\leq 3N$, let $F_i$ be the 2-factor
complementary to $M_i$.

For $1\leq i \leq 3N$, we apply Lemma~\ref{l:weights2} to the 2-factor
$F_i$ and call the resulting function $w_i$. For a total independent
set $X\in\Phi(G)$, put
\begin{equation*}
  w'(X) = \sum_{i=1}^{3N} \frac{w_i(X)}{3N}.
\end{equation*}
Since each edge of $G$ is contained in $2N$ of the factors $F_i$, each
edge gets the same weight $w'[e] = (\alpha+2\gamma)/3$. Similarly,
each vertex gets weight $w'[v] = \beta$. Observe that $w'[v] > w'[e]$
as $4\beta > 1 = \alpha + \beta + 2\gamma$. Thus, we may use the
fractional 3-edge-colouring $c$ to make the weight on edges equal to
that on vertices. Specifically, extend $c$ by setting $c(Y)=0$ for any
$Y\in\Phi(G)$ that is not a perfect matching, and define
\begin{equation*}
  w(X) = \frac1\beta \cdot w'(X) +
  \left(1-\frac{\alpha+2\gamma}{3\beta}\right)\cdot c(X).
\end{equation*}
It is easy to see that $w[x]=1$ for all $x\in V(T(G))$, so $w$ is a
fractional total colouring. Moreover, we claim that $\size w = 4$. To
see this, consider the set $\Setx{x_1,x_2,x_3,x_4}$ consisting of a
vertex of $G$ and the three adjacent edges, and note that since each
set from $\Phi(G)$ contains exactly one $x_i$, we have $\size w =
\sum_i w[x_i] = 4$. This proves Theorem~\ref{t:main} for $g\geq
15k\ell$, where the required values of $k$ and $\ell$ have been
identified in the proof of Lemma~\ref{l:weights2} as $k=11$ and
$\ell=96$. Thus, $g\geq 15\,840$ is sufficient.


\section{Subcubic graphs}
\label{sec:subcubic}

We are now ready to prove Theorem~\ref{t:main}. We show, by induction
on the order of the graph $G$, that if $G$ is a graph with maximum
degree at most 3 and girth at least $g$, then $\chi''_f(G) \leq
4$. The assertion is true for graphs with $\Delta(G) \leq 2$ by
Theorem~\ref{t:kilakos-reed} and for bridgeless cubic graphs by the
above.

Suppose first that $G$ contains a bridge $e$ with endvertices $x_1$
and $x_2$. For $i=1,2$, let $G_i$ be the component of $G-e$ containing
$x_i$. By induction, each $G_i$ has a fractional total colouring $w_i$
with $\size{w_i} \leq 4$. We may assume without loss of generality
that $\size{w_i} = 4$.

In view of Lemma~\ref{l:basic}, there is a multiset $\WW_i$ of
$4N$ total independent sets in $G_i$, such that each $x\in V(T(G_i))$
is contained in $N$ of the sets in $\WW_i$ (for a suitable integer
$N$). Let us enumerate the members of each $\WW_i$ as
$W_{i,1},\dots,W_{i,4N}$ in such a way that:
\begin{itemize}
\item $x_1$ is contained in $W_{1,1},\dots,W_{1,N}$,
\item $x_2$ is contained in $W_{2,N+1},\dots,W_{2,2N}$,
\item neither $x_i$ nor any edge incident to it are contained in
  $W_{i,j}$ for $j > 3N$.
\end{itemize}

We construct a multiset $\WW=\Setx{W_1,\dots,W_{4N}}$ of total
independent sets in $G$ by setting
\begin{equation*}
  W_j =
  \begin{cases}
    W_{1,j} \cup W_{2,j} & \text{if $j \leq 3N$},\text{ and}\\
    W_{1,j} \cup W_{2,j} \cup \Setx{e} & \text{otherwise}.
  \end{cases}
\end{equation*}
It is easy to see that each set $W_j$ is total independent and each
$x\in V(T(G))$, including $e$, is contained in $N$ of these
sets. Hence, $G$ has a fractional total colouring of size 4.

Having dealt with bridges, we may assume that $G$ is a bridgeless
subcubic graph. Let $D = \sum_{v\in V(G)} (3-d(v))$. We know that
$D>0$; assume now that $D \geq 2$. It is well-known that there exists
a $D$-regular graph $H$ with girth at least $g$; the construction
given in \cite[Solution to Problem 10.12]{bib:Lov-combinatorial}
moreover ensures that $H$ is 2-connected. Replace each vertex $w$ of
$H$ with a copy of $G$, and for each vertex $v$ of this copy, choose
$3-d(v)$ edges of $H$ formerly incident with $w$ and redirect them to
$v$. The result is a cubic bridgeless graph of girth at least
$g$. Since any fractional total 4-colouring of this graph yields a
fractional total 4-colouring of its subgraph $G$, this case is
resolved.

It remains to consider the case that $D=1$, i.e., all the vertices of
$G$ have degree 3 except for one vertex $z$ of degree 2. Let the
neighbors of $z$ be denoted by $x$ and $y$. The graph $G_z$, obtained
by suppressing $z$ (i.e., contracting one of the two edges adjacent to
$z$), is cubic and bridgeless. Let $F_1,\dots,F_{3N}$ be a multiset of
2-factors of $G_z$ such that each edge of $G_z$ is contained in
exactly $2N$ of them. We may assume that the edge $e=xy$ is contained
in $F_1,\dots,F_{2N}$.

We follow the approach of Sections~\ref{sec:algorithm} and
\ref{sec:cubic}, with modifications that we describe next.

Step I: We first process the 2-factors $F_1,\dots,F_{2N}$. We embed
$G$ in a graph $G'$ obtained as follows. Let $H'$ be a hamiltonian
cubic graph of girth at least $g$ (see~\cite{bib:Big-constructions}
for a construction) and let $S'$ be a Hamilton cycle of
$H'$. Subdivide an edge of $S'$, creating a vertex $z^*$. The graph
$G'$ is the disjoint union of $G$ and $H'$ with an added edge
$zz^*$. Note that $G'$ is cubic. It will not pose any problem that
$G'$ contains a bridge.

For $1\leq i\leq 2N$, we define $F'_i$ as the 2-factor of $G'$
corresponding to $F_i$ with the cycle $S'$ added. Using
Lemma~\ref{l:decompose} (with $E(Q)=\emptyset$), we find a
decomposition $\BB_i$ of $E(G')$ into $(F'_i,\ell)$-sparse sets. For
$B\in\BB_i$, we run the algorithm of Section~\ref{sec:algorithm} that
constructs the sets $T = T(F'_i,B)$ and $\tilde T(F'_i,B)$ without
modifications. Following the proof of Lemma~\ref{l:weights2}, we find
a function $w'_i$ defined on $\Phi(G')$, satisfying the conclusion of
that lemma with respect to the graph $G'$ and 2-factor
$F'_i$. Restricting to $G$, we obtain a function $w_i$ defined on
$\Phi(G)$ which assigns weight $\beta$ to the vertices of $G$,
$\gamma$ to edges of $G$ in $F'_i$ and $\alpha$ to the other edges of
$G$, where $\alpha$, $\beta$, $\gamma$ are the constants from
Lemma~\ref{l:weights2}.

Altogether, Step I provides us with $2N$ functions $w_1,\dots,w_{2N}$
on $\Phi(G)$ with the above property.

Step II: To process the 2-factors $F_{2N+1},\dots,F_{3N}$, we first
construct a cubic graph $H$. For some $s\geq g/2$, where $g$ is the
girth of $G$, take $s$ copies $H_1,\dots,H_s$ of $G-z$. For
$j=1,\dots,s$, let the copies of $x$ and $y$ in $H_j$ be denoted by
$x_j$ and $y_j$, and let $x'_j$ and $y'_j$ be new vertices. The graph
$H$ is obtained by taking the disjoint union of all the copies $H_j$
and the cycle $S = x'_1y'_1x'_2y'_2\dots x'_sy'_s$, and adding the
edges $x_jx'_j$ and $y_jy'_j$ for all $j=1,\dots,s$. It is easy to see
that $H$ is cubic bridgeless and its girth is at least $g$.

For each 2-factor $F_i$ of $G_z$ ($2N+1 \leq i \leq 3N$) there is a
corresponding 2-factor $F''_i$ of $H$ obtained by taking a copy of
$F_i$ in each graph $H_j$ ($1\leq j \leq s$) and adding the cycle
$S$. We aim to use Lemma~\ref{l:decompose} in $H$ to find a
decomposition of each $E(F''_i)$, $2N+1 \leq i \leq 3N$, into
$(F''_i,\ell)$-sparse sets.

As we will see, we need to ensure additionally that none of the sets
contains an edge incident with $x_j$ and another edge incident with
$y_j$ for any $j=1,\dots,s$. To this end, we apply
Lemma~\ref{l:decompose} to a graph $Q$ on $E(F''_i)$ constructed as
follows. The edge set of $Q$ contains, for each $j=1,\dots,s$, all
four edges $e_xe_y$ where $e_x$ is an edge of $F''_i$ incident with
$x_j$ and $e_y$ is an edge of $F''_i$ incident with $y_j$. Clearly,
$\Delta(Q) = 2$. Since $\ell\geq 89 = 83+3\Delta(Q)$, $Q$ may indeed
be used in Lemma~\ref{l:decompose}.

The graph $H$ is cubic, so we can run the algorithm of
Section~\ref{sec:algorithm} on it without modifications. For each
choice of a set of boundary edges $B$ (an $(F''_i,\ell)$-sparse set
obtained from Lemma~\ref{l:decompose}) and each total independent set
$\tilde T$ that the algorithm produces, we consider the total
independent set $\tilde T''$ in $G$ obtained by the following rules:
\begin{itemize}
\item each vertex and edge of $G-z$ is in $\tilde T''$ if and only if
  the corresponding vertex or edge in $H_1$ is in $\tilde T$,
\item if $x_1x'_1 \in \tilde T$, then we add $xz$ to $\tilde T''$,
\item if $y_1y'_1 \in \tilde T$, then we add $yz$ to $\tilde T''$,
\item if none of $x_1,y_1,x_1x'_1$ and $y_1y'_1$ is in $\tilde T$, we add
  $z$ to $\tilde T''$.
\end{itemize}
Each set $\tilde T''$ is total independent in $G$. To verify this, we
have to check that $\tilde T''$ does not contain both $xz$ and $yz$,
i.e., that $\tilde T$ does not contain both $x_1x'_1$ and
$y_1y'_1$. Our algorithm may add an edge of $E(H)-E(F''_i)$ to $\tilde
T$ only if the edge is incident with an edge of $B$. Since $B$ is
chosen using the above graph $Q$, this cannot happen for $x_1x_1'$ and
$y_1y_1'$ at the same time.

Based on the sets $\tilde T''$, we define the associated functions
$\map{w_i}{\Phi(G)}{[0,1]}$ (where $2N+1 \leq i \leq 3N$), obtained as
in the proof of Lemma~\ref{l:weights2}. Each $w_i$ assigns weight
$\beta$ to all vertices except $z$, $\gamma$ to all edges of $F_i$ and
$\alpha$ to all edges of $E(G)-E(F_i)$.

\begin{sloppypar}
  We need to ensure that $w_i[z] \geq \beta$. By the construction,
  $w_i[z]$ equals $\sum_{X} w_i(X)$, where $X$ ranges over total
  independent sets in $G$ containing none of $x,y,xz$ and $yz$. We
  thus have:
  \begin{equation}\label{eq:weight-z}
    w_i[z] \geq 1-w_i[x]-w_i[y]-w_i[xz]-w_i[yz] = 1-2\beta-2\alpha.
  \end{equation}
  Note that the inequality $1-2\beta-2\alpha \geq \beta$ is equivalent
  to $\gamma \geq \beta + \alpha/2$. As $\ell$ grows large, $\gamma$
  is close to $p^*$, which in turn is close to $3-\sqrt 7 \approx
  0.3542$ for large $k$ (cf. Lemma~\ref{l:sqrt}). Similarly, $\beta$
  tends to $1-2(3-\sqrt7) = 2\sqrt7 - 5 \approx 0.2915$. Furthermore,
  Lemma~\ref{l:weights2} asserts that $\alpha \leq 4/(3\ell)$, so for
  large $\ell$ and $k$ we will indeed have $\gamma \geq \beta +
  \alpha/2$. In particular, the values $k=11$ and $\ell=96$, used in
  Section~\ref{sec:cubic}, are sufficient.
\end{sloppypar}

Thus, $w_i[z] \geq \beta$. Since we may remove $z$ from any total
independent set as required, it may be assumed that $w_i[z] = \beta$. 

Following the argument at the end of Section~\ref{sec:cubic}, we can
define
\begin{equation*}
  w' = \sum_{i=1}^{3N} \frac{w_i}{3N}
\end{equation*}
and note that $w'$ assigns weight $\beta$ to each vertex and weight
$(\alpha+2\gamma)/3$ to each edge. Unfortunately, we are no longer
able to augment $w'$ to a fractional total 4-colouring using a
fractional 3-edge-colouring, since the latter need not exist in
$G$. We need to modify the proof in yet another way.

In the recurrence of Section~\ref{sec:prob}, let us replace the
equation~(\ref{eq:q}) by
\begin{equation*}
  q_k(i) = \xi \cdot p_k(i) \left( 1 - \frac1k - \frac1k\sum_{j=1}^{i-1}
    p_k(j)\right),
\end{equation*}
where $\xi$ is a real number from the interval $[0,1]$. In the
algorithm of Section~\ref{sec:algorithm}, we adjust the rule for the
inclusion of a vertex accordingly: whenever a vertex $u_j$ is to be
included by the original algorithm (that is, the events
(\ref{eq:incl-vertex1}) and (\ref{eq:incl-vertex2}) occur), we decide
with probability $1-\xi$ not to include it. With this modification,
Observation~\ref{obs:v-e-average} analyses the algorithm correctly if
we interpret $p^*$ and $q^*$ as functions of $\xi$. Similarly, let us
regard $\alpha$, $\beta$ and $\gamma$ as functions of $\xi$, so we can
write, e.g., $\beta = \beta(\xi)$. Likewise, for a function such as
$q_k(i)$ we may write $q_k(i) = q_k(i,\xi)$. Lemma~\ref{l:weights2}
remains valid, except for the assertion that $\beta > 1/4$. Indeed,
$\beta(0)$ will be small, since for $\xi=0$, the only way that a
vertex will be included in the set $\tilde T$ is through the
resolution of a conflict of type (IIIc). An argument similar to the
one used to bound $\alpha$ in Lemma~\ref{l:weights2} shows that
$\beta(0) \leq 1/(3\ell) < 1/4$.

Each of the functions $p_k(i,\xi)$ and $q_k(i,\xi)$ is easily seen to
be continuous in $\xi$. As we have observed in
Section~\ref{sec:algorithm}, the probability of a particular type of
conflict at a given edge can be expressed in terms of these functions,
and as a function of $\xi$ it will be continuous. From this it follows
that $\beta(\xi)$ is continuous, so there is a value $\eta$ for which
$\beta(\eta) = 1/4$. If we use this value in our algorithm and
construct the functions $w_i$ and $w'$ as above, each vertex $v$ will
get weight $w'[v] = 1/4$. Similarly, each edge $e$ will get weight
$(\alpha(\eta)+2\gamma(\eta))/3 = (1-\beta(\eta))/3 = 1/4$. Thus, the
function $4w$ is a fractional total colouring of weight 4. The proof
of Theorem~\ref{t:main} is complete.


\section{Graphs with even maximum degree}
\label{sec:even}

In this section, we show that with minor modifications, the method
used to prove Theorem~\ref{t:main} yields a proof of
Theorem~\ref{t:even}.

Let $G$ be a graph of maximum degree $\Delta$, where $\Delta\geq 4$ is
even. Using the method described in Section~\ref{sec:subcubic}, we
construct a $\Delta$-regular graph $H$ such that $H$ contains $G$ as a
subgraph and the girth of $H$ equals that of $G$ (at least if $G$
contains a cycle, which may be assumed without loss of generality). A
well-known result of Petersen (see, e.g.,~\cite[Corollary
2.1.5]{bib:Die-graph}) implies that $H$ can be decomposed into
edge-disjoint 2-factors $F_1,\dots,F_{\Delta/2}$ of $H$.

For each $i=1,\dots,\Delta/2$ and suitable constants $k$, $\ell$, we
use an analogue of Lemma~\ref{l:decompose} to find a decomposition
$\BB_i$ of $E(F_i)$ into $(F_i,\ell)$-sparse sets. For $B\in\BB_i$, we
then run the algorithm of Section~\ref{sec:algorithm} with a single
modification: each vertex $u$ will now have $\Delta-2$ `mates' (rather
than just one), and will only be included in the set $T(F_i,B)$ if
this set contains none of the mates whose level is lower than that of
$u$; if a mate of $u$ has the same level as $u$, then neither of them
will be included in $T(F_i,B)$. Although we can no longer use the
analysis from Section~\ref{sec:prob}, the following variant of
Lemma~\ref{l:weights2} holds:
\begin{lemma}
  \label{l:weights3}
  Let $\Delta\geq 4$ be an even integer. There are positive rational
  constants $\alpha$, $\beta$ and $\gamma$ such that
  $(\Delta-2)\alpha+\beta+2\gamma=1$ and the following holds: If $F$
  is a 2-factor of $G$, then there exists a function $\map w {\Phi(G)}
  {[0,1]}$ such that:
  \begin{equation*}
    w[x] =
    \begin{cases}
      \beta & \text{if $x\in V(G)$},\\
      \gamma & \text{if $x\in E(F)$,}\\
      \alpha & \text{if $x\in E(G)-E(F)$}
    \end{cases}
  \end{equation*}
  for all $x\in V(T(G))$.
\end{lemma}
Lemma~\ref{l:weights3} can be proved along essentially the same lines
as the corresponding part of Lemma~\ref{l:weights2}.

As $i$ ranges over $1,\dots,\Delta/2$, the average of the weights
$w[x]$ given to $x\in V(T(H))$ is
\begin{align*}
  \beta & \text{\quad if $x$ is a vertex,}\\
  \frac{(\Delta-2)\alpha + 2\gamma}\Delta & \text{\quad if $x$ is an edge.}
\end{align*}
A simple computation shows that if $\beta \geq 1/(\Delta+1)$, then the
average value for a vertex is greater than or equal to that for an
edge. In this case, we can use the argument described at the end of
Section~\ref{sec:subcubic}, modifying the equivalent of the
equation~(\ref{eq:q}) by introducing a parameter $\xi$ and using a
value of $\xi$ for which both of the above averages are equal to
$1/(\Delta+1)$. The associated probability distribution on the full
total independent sets then clearly determines a fractional total
$(\Delta+1)$-colouring of $H$ and hence of $G$.

It remains to derive the lower bound on the constant $\beta$:
\begin{proposition}\label{p:beta}
  Let $\Delta\geq 4$. In Lemma~\ref{l:weights3}, we can choose $\beta$
  in such a way that $\beta > 1/(\Delta+1)$.
\end{proposition}
\begin{proof}
  For the present setting, the recurrence of Section~\ref{sec:prob}
  changes to 
  \begin{align}
    2p_k(i) + q_k(i) &= 1,\notag\\
    q_k(i) &= p_k(i) \left( 1 - \frac1k - \frac1k\sum_{j=1}^{i-1}
      \tilde q_k(j)\right)^{\Delta-2},\label{eq:tildeq}\\
    \tilde q_k(i) &= p_k(i) \left( 1 - \frac1k -
      \frac1k\sum_{j=1}^{i-1} \tilde q_k(j)\right)^{\Delta-3},\notag
  \end{align}
  where the term $\tilde q_k(i)$ represents the probability that a
  vertex $u$ is included in the total independent set $T$ assuming
  that the level of $u$ equals $i$ and the level of a given mate of
  $u$ exceeds $i$ (for the number of levels being $k$). Following the
  method of the proof of Lemma~\ref{l:sqrt}, we define piecewise
  linear functions $h_k$ on the interval $[0,1]$ by the equations
  \begin{equation*}
    h_k\left(\frac{i-1}{k-1}\right) = \tilde q_k(i),
  \end{equation*}
  where $i=1,\dots,k$, and by the requirement that $h_k$ be linear on
  each $[\tfrac{i-1}{k-1},\tfrac{i}{k-1}]$ for $i \leq k-1$. We let
  $\map {\tilde q}{[0,1]} {[0,1]}$ be the limit of $h_k$ as
  $k\to\infty$. Thus, $\tilde q$ can be viewed as an asymptotic
  version of $\tilde q_k$. Note that in the limit, $\sum_{j=1}^{i-1}
  \tilde q_k(j)/k$ becomes $\int_0^x\!\tilde q(t)\,dt$ (for a suitable
  $x$). In accordance with~(\ref{eq:tildeq}), we set
  \begin{equation}
    q(x) = \tilde q(x) \cdot \left( 1- \int_0^x\!
    \tilde q(t)\,dt\right). \label{eq:q2}
  \end{equation}
  If we define, for $x\in [0,1]$,
  \begin{align*}
    \tilde Q(x) &= \int_0^x\! \tilde q(t)\,dt,\\
    Q(x) &= \int_0^x\! q(t)\,dt,\\
  \end{align*}
  then $Q(1)$ is the limit value of $\sum_{i=1}^k q_k(i)/k$, i.e., the
  asymptotic probability of the inclusion of a vertex in the set
  constructed by Phase I of our algorithm. It follows that to prove
  the assertion of the proposition, it suffices to prove
  \begin{equation}
    Q(1) > \frac1{\Delta+1}. \label{eq:Q1}
  \end{equation}
  This is what we do in the rest of this proof.

  Using the definition of $q_k$ in~(\ref{eq:tildeq}) and passing to
  the asymptotic form, we find that
  \begin{equation}\label{eq:q-deriv}
    \tilde Q'(x) = \frac{1-Q'(x)}2 \cdot \left( 1 - \tilde
      Q(x)\right)^{\Delta-3}.
  \end{equation}
  In this equation, $Q'(x)$ can be expressed in terms of $\tilde Q(x)$
  and its derivative using \eqref{eq:q2}:
  \begin{equation}\label{eq:derivatives}
    Q'(x) = \tilde Q'(x)(1-\tilde Q(x)).
  \end{equation}
  Substituting into \eqref{eq:q-deriv} and setting $F(x) = 1 - \tilde
  Q(x)$, we obtain the differential equation
  \begin{equation}\label{eq:deriv}
    F'(x) = -\frac{{F(x)}^{\Delta-3}}{F(x)^{\Delta-2} + 2}. 
  \end{equation}
  One can check that $F''(x)$ is positive on $[0,1]$, and that $F'(0)
  = -1/3$. Hence, $F(x) \geq 1-x/3$. This implies an upper bound on
  $F'(x)$: since the function $h(t) = -t^{\Delta-3}/(t^{\Delta-2}+2)$
  is decreasing on $[0,1]$, we obtain from~(\ref{eq:deriv}) that
  \begin{equation*}
    F'(x) \leq -\frac{(1-\frac x3)^{\Delta-3}}{(1-\frac x3)^{\Delta-2} + 2}.
  \end{equation*}
  Integrating the right hand side, we obtain
  \begin{align*}
    F(1) &= F(0) + \int_0^1\! F'(x)\,dx\\
    &\leq 1 + \int_0^1\!-\frac{(1-\frac x3)^{\Delta-3}}{(1-\frac
      x3)^{\Delta-2} + 2}\,dx\\
    &= 1+\left[\frac{3\log\left((1-\frac{x}3)^{\Delta-2}+2\right)}{\Delta-2}\right]_0^1\\
    &= 1 + \frac3{\Delta-2} \cdot \log\frac{\left(\frac23\right)^{\Delta-2} +
        2}3.
  \end{align*}
  We claim that this value does not exceed
  $\sqrt{(\Delta-1)/(\Delta+1)}$. This can be checked directly for
  $4\leq \Delta<7$. For $\Delta \geq 7$, the argument of the logarithm
  is easily seen to be at most $e^{-1/3}$, which yields
  \begin{align*}
    F(1) \leq 1 - \frac1{\Delta-2} < \sqrt{\frac{\Delta-1}{\Delta+1}}.
  \end{align*}
  
  By~(\ref{eq:derivatives}) and the fact that $Q(0) = \tilde Q(0) =
  0$, we have
  \begin{equation*}
    Q(x) = \tilde Q(x) - \frac{\tilde Q(x)^2}2 = \frac{1-F(x)^2}2,
  \end{equation*}
  so the above upper bound on $F(1)$ implies
  \begin{equation*}
    Q(1) > \frac{1-\frac{\Delta-1}{\Delta+1}}2 = \frac1{\Delta+1},
  \end{equation*}
  proving the desired inequality~(\ref{eq:Q1}). 
\end{proof}

Note that (despite the technicalities in the proof of
Proposition~\ref{p:beta}) the argument for graphs with even maximum
degree is simpler than that for subcubic graphs in that it works with
just a decomposition of the graph into 2-factors, without the need to
use a uniform cover by 1-factors as in Theorem~\ref{t:edge}. For
graphs with odd maximum degree $r$, however, it is not clear how to
proceed without a suitable analogue of
Theorem~\ref{t:edge}. Furthermore, the natural analogue of
Theorem~\ref{t:edge} for $r$-regular graphs ($r$ odd) does not hold in
general. Still, it seems plausible that the following is true:

\begin{conjecture}\label{conj:odd}
  The conclusion of Theorem~\ref{t:even} holds for graphs with odd
  maximum degree as well.
\end{conjecture}

So far, we have only been able to verify Conjecture~\ref{conj:odd} for
the case of \emph{$r$-graphs} ($r$-regular graphs with no odd
edge-cuts of size smaller than $r$).

\section*{Acknowledgments}

We thank Jean-S\'{e}bastien Sereni and two anonymous referees for
their careful reading of the manuscript and a number of suggested
corrections. 


\end{document}